\documentclass[11pt]{amsart}
\usepackage{enumerate,url,amssymb, mathrsfs}
\usepackage[thinlines]{easybmat}
\newtheorem{theorem}{Theorem}[section]
\newtheorem{lemma}[theorem]{Lemma}
\newtheorem{proposition}[theorem]{Proposition}
\newtheorem*{MRes}{Main Result}

\theoremstyle{definition}
\newtheorem{definition}[theorem]{Definition}

\newtheorem{corollary}[theorem]{Corollary}

\theoremstyle{remark}

\newtheorem{question}[theorem]{Question}
\newtheorem*{EP}{Extension Problem}

\numberwithin{equation}{section}

\newcommand{\abs}[1]{\lvert#1\rvert}
\newcommand{\norm}[1]{\lVert#1\rVert}
\newcommand{\inn}[1]{\langle#1\rangle}
\newcommand{\B}{\mathbb{B}}
\newcommand{\R}{\mathbb{R}}
\newcommand{\QC}{\mathcal{QC}}
\newcommand{\BL}{\mathcal{BL}}
\newcommand{\DM}{ \mathcal{DM}}

\DeclareMathOperator{\diam}{diam}

\DeclareMathOperator{\loc}{loc}

\def\XXint#1#2#3{{\setbox0=\hbox{$#1{#2#3}{\int}$}
\vcenter{\hbox{$#2#3$}}\kern-.5\wd0}}

\begin{document}

\title[An $N$-dimensional version of  the Beurling-Ahlfors extension]{An $N$-dimensional version of  \\the Beurling-Ahlfors extension}

\author{Leonid V. Kovalev}
\address{Department of Mathematics\\ Syracuse University\\ Syracuse,
NY 13244, USA}
\email{lvkovale@syr.edu}
\thanks{Kovalev was supported by the NSF grant DMS-0913474.}

\author{Jani Onninen}
\address{Department of Mathematics\\ Syracuse University \\ Syracuse,
NY 13244, USA}
\email{jkonnine@syr.edu}
\thanks{Onninen was supported by the NSF grant  DMS-0701059.}


\date{April 17, 2008}


\begin{abstract}
We extend monotone quasiconformal mappings from dimension $n$ to \mbox{$n+1$} while  preserving both monotonicity and quasiconformality. The extension is given explicitly by an integral operator. In the case $n=1$ it yields a refinement of the Beurling-Ahlfors extension.
\end{abstract}

\maketitle

\section{Introduction}

\begin{EP}
Given a mapping $f \colon \R^n \to \R^n$ of class $\mathscr A$, find  $F \colon \R^{n+1} \to \R^{n+1}$ of class $\mathscr A$ such that the restriction of $F$ to $\R^n$ agrees with $f$.
\end{EP} 

Let us introduce  coordinate notation $x=(x^1, \dots ,x^n)$ and $f=(f^1, \dots , f^n)$. By setting $F^i=f^i$ for $i=1, \dots , n$ and $F^{n+1}=x^{n+1}$ one immediately obtains a solution to the extension problem for many classes $\mathscr A$ such as continuous ($\mathscr A=C^0$), smooth ($\mathscr A=C^{k}$), homeomorphic, diffeomorphic, and (bi-)Lipschitz mappings.

When $\mathscr A=\QC$, the class of quasiconformal mappings, the extension problem is much more difficult. It was solved 
\begin{itemize}
\item for $n=1$ by Beurling and Ahlfors \cite{BA} in 1956,
\item  for $n=2$ by Ahlfors \cite{Ah} in 1964,
\item for $n\le 3$ by Carleson \cite{Ca} in 1974, and
\item for all $n \ge 1$ by Tukia and V\"ais\"al\"a \cite{TV} in 1982.
\end{itemize}

The Tukia-V\"ais\"al\"a extension uses, among other things, Sullivan's   theory~\cite{Su} of deformations of Lipschitz embeddings. Our goal is to give an explicit extension for a subclass of $\QC$. Quasiconformal mappings can be defined as orientation-preserving quasisymmetric  mappings \cite{Heb, Vab}.

\begin{definition}\label{qs}
A homeomorphism $f \colon \R^n \to \R^n$ is quasisymmetric if there is a homeomorphism $\eta \colon [0,\infty ) \to [0, \infty )$ such that
\begin{equation}\label{Kav24}
 \frac{|f(x)-f(z)|}{|f(y)-f(z)|} \le \eta \left(\frac{|x-z|}{|y-z|}\right) .
\end{equation}
for $x,y, z \in \R^n$, $z \ne y$.
\end{definition}

One can say  that quasisymmetry is a three-point condition. But there are two   subclasses of  $\QC$ that are defined by \emph{two-point} conditions, namely bi-Lipschitz class  $\BL$ and the class of
nonconstant delta-monotone mappings~\cite[Chapter 3]{AIMb}.  Recall that a mapping $f\colon \R^n \to \R^n$ is  {\it monotone} if
\begin{equation}\label{mon}
\inn{f(x)-f(y), x-y } \ge  0 \qquad \textnormal{ for all } x, y\in \R^n.
\end{equation}
We called  $f$ {\it delta-monotone} if there exists  $\delta >0$ such that
 \begin{equation}\label{demon}
\inn{f(x)-f(y), x-y } \ge   \delta\abs{f(x)-f(y)} \abs{x-y}  \qquad \textnormal{ for all } x, y\in \R^n.
\end{equation}
The class of nonconstant delta-monotone mappings is denoted by  $\DM$. 
When we want to specify  the value of $\delta$ we write that $f$ is $\delta$-monotone.

In contrast to the bi-Lipschitz case, the
extension problem for the class $\DM$ cannot be solved by means of the
trivial extension. For example, the mapping $f(x)=\abs{x}^p x$,
$p>-1$, belongs to $\DM$ but its trivial extension does not (unless
$p=0$).

\begin{MRes} Let $n \ge 2$. For any mapping $f \colon \R^n \to \R^n$ of class $\DM$ there exists  $F \colon \R^{n+1} \to \R^{n+1}$ of class $\DM$ such that the restriction of $F$ to $\R^n$ agrees with $f$.
\end{MRes}

Our proof is by an explicit construction that can be viewed as an $n$-dimensional version of the Beurling-Ahlfors extension. Suppose $f\in \DM$. Let $\R^{n+1}_+=\R^n\times[0,\infty)$ and 
 \begin{equation}\label{gauss}
 \phi(x) =(2 \pi)^{-\frac{n}{2}} e^{-\abs{x}^2/2}, \qquad x\in \R^n.
 \end{equation} 
We define $F \colon \R^{n+1}_+ \to  \R^{n+1}_+ $ by
\begin{eqnarray}
F^i(x,t)&= &\int_{\R^n} f^i(x+ty)\,\phi (y)\, d y  \qquad i=1, \dots , n \label{Fi}\\
F^{n+1}(x,t)&= &\int_{\R^n}  \inn{f(x+ty),y}\, \phi(y)\, dy \label{Fn1} 
 \end{eqnarray}
where $x\in\R^n$, $t\ge 0$ (see \S\ref{pfth} for the convergence of these integrals). 
Observe that $F(x,0)=(f(x),0)$. Furthermore, $F^{n+1}(x,t) \ge 0$ because 
\[\int_{\R^n}  \inn{f(x+ty),y}\, \phi(y)\, dy =  \int_{\R^n}  \inn{f(x+ty) -f(x),y}  \, \phi(y)\, dy \ge 0  \]
due to the monotonicity of $f$. Finally, we extend $F$ to $\R^{n+1}$ by reflection
\[F^i (x,t) = F^i(x,-t) \quad i=1,\dots , n \quad \mbox{ and } \quad  F^{n+1}(x,t)= - F^{n+1}(x,-t).\]

\begin{theorem}\label{Thm}
Let $n \ge 2$.  If $f \colon \R^n \to \R^n$  is  $\delta$-monotone, then  $F \colon \R^{n+1} \to  \R^{n+1} $ is   $\delta_1$-monotone   where $\delta_1$ depends only on $\delta$ and $n$. In addition, $F \colon \mathbb H^{n+1} \to \mathbb H^{n+1}$ is bi-Lipschitz in the hyperbolic metric.
\end{theorem}

Here $\mathbb H^{n+1} =\R^n\times(0,\infty) $ and the hyperbolic metric on $\mathbb H^{n+1}$ is $\abs{d x}/x^{n+1}$. 
Theorem~\ref{Thm} can be also formulated for $n=1$, in which case it becomes  a refinement of the  Beurling-Ahlfors extension theorem.  

\begin{proposition}\label{n1} 
If $f \colon \R \to \R$  is  increasing and quasisymmetric, then  $F \colon \R^{2} \to  \R^{2} $ is $\delta_1$-monotone where $\delta_1$ depends only on $\eta$ in Definition~\ref{qs}. Furthermore, $F \colon \mathbb H^{2} \to \mathbb H^{2}$ is bi-Lipschitz in the hyperbolic metric.
\end{proposition}

Fefferman,  Kenig and  Pipher \cite[Lemma 4.4]{FKP} proved that $F$ in Proposition~\ref{n1} is quasiconformal. Proposition~\ref{n1} was originally proved in~\cite{Ko} using their result. In this paper we give a direct proof.

Theorem~\ref{Thm} has an application to mappings with a convex
potential~\cite{Ca92}, i.e., those of the form $f=\nabla u$ with $u$
convex. The basic properties and examples of quasiconformal mappings
with a convex potential are given in~\cite{KM}. 

\begin{corollary}\label{convex} Suppose that $f\colon \R^n\to\R^n$,
$n\ge 2$, is a $K$-quasiconformal
mapping with a convex potential. Then $f$ can be extended to a
$K_1$-quasicon\-formal mapping $F\colon \R^{n+1}\to\R^{n+1}$ with a
convex potential, where $K_1$ depends only on $K$ and $n$.
\end{corollary}

\section{Preliminaries}\label{secdef}
Let $e_1,\dots,e_{n+1}$ be the standard basis of $\R^{n+1}$. All vectors are treated as  column vectors. The transpose of a vector $v$ is denoted by $v^T$. We use the operator norm $\norm{\cdot}$ for matrices.
  A Borel measure $\mu$ on $\R^n$ is  \emph{doubling} if there exists $\mathscr D_\mu$, called the doubling constant of $\mu$, such that
\[\mu (2B) \le \mathscr D_\mu \, \mu (B)\] 
for all balls $B=B(x,r)$. Here $2B=B(x,2r)$.

The geometric definition of   class $\QC$ given in the introduction is equivalent to the following analytic definition \cite{Heb, Vab}.
\begin{definition}\label{defKqc}
A homeomorphism $f \colon \R^n \to \R^n$ ($n \ge 2$) is quasiconformal if $f \in W^{1,n}_{\loc} (\R^n , \R^n)$ and there exists a constant $K$ such  that the differential matrix $Df(x)$ satisfies the distortion inequality
\[\norm{Df(x)}^n \le K \det Df(x) \quad \mbox{ a.e. in } \R^n. \]
\end{definition}

Delta-monotone mappings also have  an   analytic definition. 

\begin{lemma}\label{deltas}
Let $\Omega$ be a convex domain in $\R^n$, $n \ge 2$. Suppose   $f\in W^{1,1}_{\loc} (\Omega , \R^n)$  is  continuous. The following are equivalent:
\begin{enumerate}[(i)]
\item\label{delta1} $f$ is $\delta$-monotone in $\Omega$ for some $\delta >0$; that is,~\eqref{demon} holds for all $x,y\in\Omega$;
\item\label{delta2} there exists $\delta>0$ such that for a.e. $x\in \Omega$ the matrix $Df(x)$ satisfies 
\begin{equation*}%
v^TDf(x)v \ge \delta \abs{Df(x)v} \abs{v}\qquad \text{for every vector $v\in \R^n$;}
\end{equation*}
\item\label{delta3} there exists $\gamma>0$ such that for a.e. $x\in\Omega$ the matrix $Df(x)$ satisfies 
\begin{equation*}%
v^TDf(x)v \ge \gamma \norm{Df(x)} \abs{v}^2\qquad \text{for every vector $v\in \R^n$.}
\end{equation*}
\end{enumerate}
The constants $\delta$ and $\gamma$ depend only on each other.
\end{lemma}

\begin{proof}
The equivalence of~\eqref{delta1} and~\eqref{delta2}, with the same constant $\delta$, was proved in~\cite[p. 397]{Ko}. It is obvious that~\eqref{delta3} implies~\eqref{delta2} with $\delta=\gamma$. It remains to establish the converse implication ~\eqref{delta2}$\implies$~\eqref{delta3}. To this end we need the following \\
{\bf Claim:} if a real square matrix $A$ satisfies
\begin{equation*}%
v^TAv \ge \delta \abs{Av} \abs{v}\qquad \text{for every $v\in \R^n$}
\end{equation*}
then 
\begin{equation}\label{something}
\abs{Av} \ge c\norm{A}\abs{v}\qquad c=c(\delta)>0.
\end{equation}
Although   this claim is known, even with a sharp constant~\cite{AIM}, we give a proof for the sake of completeness.  It suffices to estimate $\abs{Av}$ from below under the assumptions that $Av \ne 0$ and $\norm{A}=1=\abs{v}$.  Let $u$ be a unit vector in $\R^n$ such that $\abs{Au}=1$. Replacing $u$ by $-u$ if necessary we may assume that $u^TAv+v^TAu\le0$. Let $\lambda = \sqrt{\abs{Av}}$. On one hand we have
\begin{equation}\label{no1}
(\lambda u +v)^T A (\lambda u +v) \le \lambda^2 u^T Au + v^T Av \le \lambda^2 + \lambda^2=2\lambda^2.
\end{equation}
On the other hand
\begin{equation}\label{no2}
(\lambda u +v)^T A (\lambda u +v) \ge  \delta \abs{\lambda Au + Av} \abs{\lambda u +v} \ge \delta (\lambda - \lambda^2)(1-\lambda).
\end{equation}
Combining~\eqref{no1} and~\eqref{no2} we obtain
$2 \lambda \ge    \delta (1-\lambda)^2$, hence
\[\lambda \ge \delta^{-1}+1- \sqrt{(\delta^{-1}+1)^2-1}>0.\]
This proves the claim.
\end{proof}

\section{Delta-monotone mappings and doubling measures}

The following result shows that $\DM\subset \QC$. In particular, $f\in \DM$ implies that $f$ is a continuous 
Sobolev mapping, and therefore~\eqref{delta2}--\eqref{delta3} of Lemma~\ref{deltas} hold.
\begin{proposition}\label{dmqs} \cite[Theorem 6]{Ko}  Every nonconstant $\delta$-monotone mapping is   $\eta$-quasisymmetric where $\eta$ depends only on $\delta$.
\end{proposition}

It is well-known that quasisymmetric mappings are closely related to doubling measures~\cite{Heb}. The following lemma is another instance of this relation.

\begin{lemma}\label{Dfdoub}
For any nonconstant $\delta$-monotone mapping $f \colon \R^n \to \R^n$ ($n\ge 2$) the measure $\mu= \norm{Df(x)} \, dx$ is  doubling. The doubling constant $\mathscr D_\mu$ depends only on $\delta$ and $n$. 
\end{lemma}

\begin{proof}
Recall that $f$ is quasisymmetric. Lemma 3.2 in \cite{KMW} implies the existence of a constant $C=C(\delta, n)$ such that
\begin{equation}\label{box}
 C^{-1} \frac{\diam f(B)}{\diam B} \le \frac{1}{|B|} \int_B \norm{Df} \, d x \le C \frac{\diam f(B)}{\diam B} 
\end{equation}
for all balls $B \subset \R^n$. Since $\diam f(2B) \le C \diam f(B)$ with $C=C(\eta)$, the lemma follows.
\end{proof}

Recall that $\phi\colon\R^n\to (0,\infty)$ is the Gaussian kernel~\eqref{gauss}. Let $\B=B(0,1)$ be the open unit ball in $\R^n$.

\begin{lemma}\label{doublem}
Let $\mu$ be a doubling measure in $\R^n$ and  $p \ge 0$. Let $\Omega$ be either $\R^n$ or  the half space $\{y \colon \inn{y, \xi} \ge 0\}$ for some  $\xi \in \R^n$.  Then
\begin{equation}\label{doubineq}
C^{-1} \mu \big( \B \big) \le \int_{\Omega} \abs{y}^p \phi(y)\, d \mu (y) \le C   \mu \big( \B\big) 
\end{equation}
where the  constant $C$   depends only on $\mathscr D_\mu$, $p$ and $n$.
\end{lemma}
\begin{proof}
We begin by estimating the integral in~\eqref{doubineq} from above as follows
\[\int_{\R^n} \abs{y}^p \phi(y)\, d \mu (y)  = \int_{\B}  \abs{y}^p  \phi(y)\, d \mu (y)  + \sum_{k=0}^\infty    \int_{2^{k} < |y| \le 2^{k+1}} \abs{y}^p  \phi(y)\, d \mu (y),
\]
where 
\[\int_{\B}  \abs{y}^p \phi(y)\, d \mu (y) \le \phi(0) \mu(\B) = (2 \pi)^{-\frac{n}{2}} \mu(\B) \]
and 
\[
\begin{split}
 \int_{2^{k} < |y| \le 2^{k+1}}  \abs{y}^p \phi(y)\, d \mu (y) & \le  2^{p(k+1)}(2 \pi)^{-\frac{n}{2}} e^{-2^{2k-1}}  \mu(B(0,2^{k+1})) \\
 & \le 2^{p(k+1)} (2 \pi)^{-\frac{n}{2}} e^{-2^{2k-1}} \mathscr D_\mu^{k+1} \mu(\B). 
  \end{split}
 \]
Summing over $k=0,1, 2 \dots$ we obtain
\[
\int_{\R^n}  \phi(y)\, d \mu (y) \le    {C} \,  \mu(\B)
\]
where $C=C(\mathscr D_\mu, p, n)>0$. 

We turn to the left side of~\eqref{doubineq}. The inequality
\[
\abs{y}^p \phi(y) \ge  \frac{e^{-1/2}}{2^p (2 \pi)^{n/2}} \qquad \mbox{ for } \quad \frac{1}{2}\le \abs{y} \le 1
 \]
implies
\[
\int_{\Omega} \abs{y}^p \phi(y)\, d \mu (y) \ge \frac{e^{-1/2}}{2^p (2 \pi)^{n/2}} \mu (\Omega \cap \{  1/2\le \abs{y} \le 1 \}). 
\]
Since  $ \mu (\Omega \cap \{  1/2\le \abs{y} \le 1 \})  \ge \mathscr D_\mu^{-3} \mu(\B)$, the left side of~\eqref{doubineq} follows.
\end{proof}

\section{Proof of main results}\label{pfth}
\begin{proof}[Proof of Theorem~\ref{Thm}]
Since $f$ is quasisymmetric by Proposition~\ref{dmqs}, it satisfies the growth condition $\abs{f(x)} \le \alpha \abs{x}^p + \beta$ for some constants $\alpha, \beta,  p $, see \cite[Theorem 11.3]{Heb}. 
Therefore, the integrals~\eqref{Fi} and~\eqref{Fn1} converge and $F$ is $C^\infty$-smooth in  $\mathbb H^{n+1}$. Let $\gamma=\gamma(\delta)>0$ be as in part~\eqref{delta3} of Lemma~\ref{deltas}. 

Our first step is to prove that for $(x,t)\in\mathbb H^{n+1}$ the matrix $\mathscr B:=DF(x,t)$ satisfies the condition
\begin{equation}\label{final}
w^T\mathscr B w \ge \gamma_1 \norm{\mathscr B } \abs{w}^2\qquad \text{for every vector $w\in \R^{n+1}$}
\end{equation}
where $\gamma_1=\gamma_1(\delta,n)>0$. Fix $x\in \R^n$ and $t>0$.  We compute the  partial derivatives of $F$ at $(x,t)\in \mathbb H^{n+1}$ as follows.
\[
\begin{split}
\frac{\partial F^i}{\partial x_j}  &= \int_{\R^n} f_j^i (x+ty) \phi(y) d y ,\quad 1 \le i,j \le n; \\
\frac{\partial F^i}{\partial t} &= \int_{\R^n}  \sum_{j=1}^n  f_j^i (x+ty) y^i \phi(y) d y ,\quad 1 \le i \le n; \\
\frac{\partial F^{n+1}}{\partial x_j}   &=  \int_{\R^n} \sum_{i=1}^n   f_j^i (x+ty) y^j \phi(y) d y ,\quad 1 \le j \le n; \\
\frac{\partial F^{n+1}}{\partial t}  &= \int_{\R^n}  \sum_{i=1}^n \sum_{j=1}^n  f_j^i (x+ty) y^iy^j \phi(y) d y .
 \end{split}
\]
To simplify formulas we write $A(y)=Df(x+ty)$ and let $B(y)$ be the $(n+1)\times (n+1)$ matrix
written in block form below.
\begin{equation}\label{Bmatrix} B(y)=\left(
   \begin{BMAT}(@,50pt,20pt){c.c}{c.c}
      A(y) & A(y)y \\ y^T A(y)& y^T A(y) y
   \end{BMAT}
   \right) .\end{equation}
With this notation we have
\begin{equation}\label{afterBmatrix}
DF(x,t) = \int_{\R^n} B(y) \phi(y) \, dy .
\end{equation}
First we show that   the norm of $\mathscr B$ is dominated by the quantity 
\[
\alpha := \int_{B(0,1)} \norm{A(y)}   \, dy .  
\]
Indeed, 
\[
\norm{\mathscr B} \le  \int_{\R^n} \norm{B(y)} \phi(y) \, dy \le    \int_{\R^n} \norm{A(y)} (1+\abs{y})^2\phi(y) \, dy.  
\]
By Lemma~\ref{Dfdoub} the measure $\mu = \norm{A(y)}\,  dy$ is doubling. Applying Lemma~\ref{doublem}  we obtain 
\begin{equation}\label{DF1}
\norm{\mathscr B} \le  C \alpha, \qquad C=C(\delta, n) .
\end{equation}
Next we estimate the quadratic form  $w\mapsto w^T\mathscr B w$ generated by $\mathscr B$ from below. For this we fix a vector $w \in \R^{n+1}$, written as $w= v+se_{n+1}$ with $v \in \R^n$ and $s\in \R$.  It is easy to see that
\[w^T B(y) w  = (v+sy)^T A(y) (v+sy).\]
Let $\Omega = \{y \in \R^n \colon \inn{v,sy} \ge 0\}$.
Then
\[ 
\begin{split}
w^T\mathscr B w& = \int_{\R^n} \left\{  (v+sy)^T A(y) (v+sy)  \right\} \phi(y) \, dy \\
& \ge \gamma \int_{\R^n} \norm{A(y)} \abs{v+sy}^2 \phi(y) \, dy \\
& \ge \gamma \int_{\Omega} \norm{A(y)} \abs{v+sy}^2 \phi(y) \, dy \\
& \ge \gamma \abs{v}^2 \int_{\Omega} \norm{A(y)}   \phi(y) \, dy +   \gamma s^2 \int_{\Omega} \norm{A(y)} \abs{y}^2  \phi(y) \, dy. 
\end{split}
\] 
Applying Lemma~\ref{doublem} with $\mu = \norm{A(y)}\,  dy$  we obtain
\begin{equation}\label{DF2}
w^T \mathscr B w \ge c\, \alpha  \gamma  (\abs{v}^2 + s^2)  = c \,  \alpha \gamma \abs{w}^2,  \qquad c=c(\delta, n) .
\end{equation}
Combining~\eqref{DF1} and~\eqref{DF2} we obtain~\eqref{final} with $\gamma_1=(c/C) \gamma $.     By virtue of Lemma~\ref{deltas} $F$ is $\delta_1$-monotone in the upper half-space $\mathbb H^{n+1}$ where $\delta_1=\delta_1(\delta,n)$. By symmetry,  $F$ is also $\delta_1$-monotone in the lower half-space.

To prove that $F$ is $\delta_1$-monotone in the entire space $\R^{n+1}$, we consider two points $a,b\in\R^{n+1}$ such that the line segment $[a,b]$ crosses the hyperplane $\R^n$
at some point $c$. We have
\[
\begin{split}
\inn{F(a)-F(b), a-b}& = \inn{f(a)-f(c), a-b} + \inn{F(c)-F(b), a-b}\\
& \ge \delta_1 \abs{F(a)-F(c)} \abs{a-b} + \delta_1 \abs{F(c)-F(b)} \abs{a-b}\\
& \ge \delta_1 \abs{F(a)-F(b)} \abs{a-b}
\end{split}
\]
Therefore, $F\in \DM$. 

It remains to show that  $F \colon \mathbb H^{n+1} \to \mathbb H^{n+1}$ is bi-Lipschitz in the hyperbolic metric.  
Since $F\in \QC$ and $\mathbb H^{n+1}$ is a geodesic space, it suffices to prove that
\begin{equation}\label{name}
 \norm{DF(x,t)} \approx \frac{F^{n+1} (x,t)}{t}.
\end{equation}
Here  $X \approx Y$ means that $X$ and $Y$ are comparable, i.e., $C^{-1} Y \le X \le CY$ where $C=C(\delta, n)$.
It follows from~\eqref{DF1} and~\eqref{DF2} that $\norm{DF(x,t)}$ is comparable to the integral average of $\norm{Df}$ over the ball $B(x,t)$. By~\eqref{box} this average is comparable to $t^{-1}\diam f(B(x,t))$.  The  quasisymmetry of $F$ implies (cf. \cite[11.18]{Heb})
\[\diam f(B(x,t)) \approx \abs{F(x,t)- F(x, t/2)} \approx F^{n+1}(x,t) .\]
This proves~\eqref{name}.
\end{proof}

\begin{proof}[Proof of Proposition~\ref{n1}]
The proof of Theorem~\ref{Thm} also works in the case $n=1$ with the following interpretation. 
Since quasisymmetric mappings on the line need not be  absolutely continuous~\cite{BA}, the derivative $f'$ must be understood in the sense of distributions. In fact, $\mu:=f'$ is a positive doubling measure with
$\mathscr D_{\mu}=\mathscr D_{\mu}(\eta)$ \cite[13.20]{Heb}.
Lemma~\ref{Dfdoub} is not needed in this case.
The rest of the proof carries over with $\gamma=1$ and $\gamma_1=\gamma_1(\mathscr D_{\mu})$.
\end{proof}

\begin{proof}[Proof of Corollary~\ref{convex}]
According to~\cite[Lemma 18]{Ko}, a $K$-quasiconformal mapping with a
convex potential is also $\delta$-monotone with $\delta=\delta(K,n)$.
Let $F$ be the $\delta_1$-monotone extension of $f$ provided by
Theorem~\ref{Thm}. Since the differential matrix $Df$ is symmetric,
the formulas~\eqref{Bmatrix} and~\eqref{afterBmatrix} show that $DF$ is
symmetric as well. In addition, $DF$ is positive semidefinite by
Lemma~\ref{deltas}. Thus, $F=\nabla U$ for some convex function
$U\colon \R^{n+1}\to \R$.
\end{proof}

\section{Concluding remarks}

Both classes $\QC$ (quasiconformal) and $\BL$ (bi-Lipschitz) are groups under composition. However, the class of delta-monotone mappings $\DM$  is not closed under composition (consider the rotation of the complex plane given by $z \mapsto e^{i \theta} z$ where $\abs{\theta} < \pi /2$). Let $\QC_d \subset \QC$ be the group generated by $\BL$ and $\DM$. In other words, $f$ belongs to $\QC_d$ if it can be decomposed into bi-Lipschitz and delta-monotone mappings. This should be compared with the notion of polar factorization of mappings introduced by Brenier~\cite{Br}.

Theorem~\ref{Thm} together with the trivial extension of bi-Lipschitz mappings yield a solution to the extension problem for $\QC_d$. 

\begin{corollary}\label{cor} Let $n \ge 2$. For any mapping $f \colon \R^n \to \R^n$ of class $\QC_d $ there exists  $F \colon \R^{n+1} \to \R^{n+1}$ of class $\QC_d $ such that the restriction of $F$ to $\R^n$ agrees with $f$.
\end{corollary}

It seems likely that $\QC_d $ is a proper subset of $\QC$. This motivates the following question:

\begin{question}\label{which}
Which quasiconformal mappings are decomposable? 
\end{question}

Both bi-Lipschitz and delta-monotone mappings take smooth curves into rectifiable curves~\cite[Theorem 3.11.7]{AIMb}. This is no longer true for their composition. More precisely, for any $1<\alpha<2$ one can construct a mapping
$f\colon \R^2\to\R^2$ such that $f \in \QC_d$ and $f(\R)$ has Hausdorff dimension at least $\alpha$. To this end, one first finds a bi-Lipschitz mapping $g\colon \R^2\to\R^2$ such that $g(\R)$ contains a planar Cantor set $E$ of dimension $0<\beta<1$ (see Lemma~3.1\cite{Bi} and the comment after its proof). Second, there is a delta-monotone mapping $h\colon \R^2\to\R^2$ such that the Hausdorff dimension of $h(E)$ is equal to $\alpha$ (see the construction in~\cite[Theorem~5]{GV}). Finally, let $f=h\circ g$. 

\section*{Acknowledgments}
We thank Mario Bonk  and Jang-Mei Wu for conversations related to the subject of this paper.

\bibliographystyle{amsplain}

\end{document}